\documentclass[11pt,twoside]{article}
\usepackage{multicol}
\usepackage{caption}
\usepackage{graphicx,tikz}
\usepackage{amsthm,amssymb,amsmath,mathrsfs,bbm,stmaryrd}
\usepackage [bindingoffset=-.5cm, footskip=1cm, headheight = 16pt, top=3cm, bottom=2.5cm,  right=3cm,  left=3cm ] {geometry}
\usepackage{graphicx,xcolor,wrapfig}
\usepackage[pagebackref=true,colorlinks,linkcolor=blue,citecolor=magenta,urlcolor=cyan]{hyperref}
\usepackage{fancyhdr}
\usepackage{prettyref,cite}
\usepackage[nottoc]{tocbibind}
\usepackage[textfont=it]{caption}
\usepackage{makeidx}
\def\Cchi{{\raisebox{.2ex}{ \large $\chi$}}}
\newtheorem{prethm}{{\bf Theorem}}

\newenvironment{thm}{\begin{prethm}{\hspace{-0.5
em}{\bf~}}}{\end{prethm}}
\newtheorem{precor}{{\bf Corollary}}

\newtheorem{preprop}{{\bf Proposition}}

\newtheorem{preque}{{\bf Question}}

\newtheorem{preques}{{\bf Question}}

\newtheorem{preremark}{{\bf Remark}}

\newtheorem{prelemma}{{\bf Lemma}}

\newenvironment{lemma}{\begin{prelemma}{\hspace{-0.5
em}{\bf~}}}{\end{prelemma}}
\newtheorem{prelemm}{{\bf Lemma}}

\newtheorem{preex}{{\bf Example}}

\newenvironment{ex}{\begin{preex}{\hspace{-0.5em}{\bf~}}}{\end{preex}}
\newtheorem{prepro}{{\bf Proposition}}

\newtheorem{preobs}{{\bf Observation}}

\newtheorem{preprob}{{\bf Problem}}

\newtheorem{prelem}{{\bf Theorem}}

\newtheorem{preproof}{{\bf Proof.}}

\newtheorem{preconj}{{\bf Conjecture}}

\newtheorem{predefi}{{\bf Definition}}

\newenvironment{defi}{\begin{predefi}{\hspace{-0.90
em}{\bf~}}}{\end{predefi}}
\newtheorem{predeff}{{\bf Definition}}

\date{}
\title{{\large\bf \vspace*{-5mm} A Bound for the Locating Chromatic Numbers of Trees}}

{
\small
\author{
{\sc Ali Behtoei\footnote{{\small \it This research was in part supported by a grant from IPM (No. 91050012).}}}~~~~and~~~~{\sc Mahdi Anbarloei}\\
{\small \it a.behtoei@sci.ikiu.ac.ir}\\
{\small \it  Department of Mathematics}\\
{\small \it  Imam Khomeini International University} \\
{\small \it  P.O. Box: 34149-16818, Qazvin, Iran}\\
~and \\
{\small \it School of Mathematics} \\
{\small \it Institute for Research in Fundamental Sciences (IPM)} \\
{\small \it P.O. Box: 19395-5746, Tehran, Iran}
}
}
\begin{document}

%%%%%%%%%%%%%%%%%%%
\small
\voffset=-20mm
\maketitle
%%%%%%%%%%%%%%%%%%%

\begin{abstract}
{\footnotesize Let $f$ be a proper $k$-coloring of a connected graph $G$  and
$\Pi=(V_1,V_2,\ldots,V_k)$ be an ordered partition of $V(G)$ into
the resulting color classes. For a vertex $v$ of $G$, the color
code of $v$ with respect to $\Pi$ is defined to be the ordered
$k$-tuple $c_{{}_\Pi}(v)=(d(v,V_1),d(v,V_2),\ldots,d(v,V_k)),$
where $d(v,V_i)=\min\{d(v,x):~x\in V_i\}, 1\leq i\leq k$. If
distinct vertices have distinct color codes, then $f$ is called a
locating coloring.  The minimum number of colors needed in a
locating coloring of $G$ is the locating chromatic number of $G$,
denoted by  $\Cchi_{{}_L}(G)$. In this paper, we study the locating chromatic numbers of trees.
We provide a counter example to a theorem of Gary Chartrand et al.
[G. Chartrand, D. Erwin, M.A. Henning, P.J. Slater, P. Zhang, The locating-chromatic number of a graph,
Bull. Inst. Combin. Appl.  36 (2002) 89-101]
about the locating chromatic numbers of trees. Also, we offer a new bound for the locating chromatic number of trees. Then, by constructing a special family of trees, we show that this bound is best possible.}
\end{abstract}

%\bigskip

\section{Introduction}

 Let $G$ be a graph
without loops and multiple edges with vertex set $V(G)$ and edge
set $E(G)$. A proper $k$-coloring of $G$, $k\in \Bbb{N}$,  is a
function $f$ defined from $V(G)$ onto a set of colors
$[k]=\{1,2,\ldots,k\}$ such that every two adjacent vertices have
different colors. In fact, for every $i$, $1\leq i\leq k$, the set
$f^{-1}(i)$ is a nonempty independent set of vertices which is
called the color class $i$. When $S\subseteq V(G)$, then $f(S)=\{f(u):~u\in S\}$. The minimum cardinality $k$ for which
$G$ has a proper $k$-coloring is the chromatic number of $G$,
denoted by $\Cchi(G)$. For a connected graph $G$, the distance
$d(u,v)$ between two vertices $u$ and $v$ in $G$ is the length of
a shortest path between them, and for a subset $S$ of $V(G)$, the
distance between $u$ and $S$ is given by
$d(u,S)=\min\{d(u,x):~x\in S\}$. The diameter of $G$ is $\max\{d(u,v):~u,v\in V(G)\}$. When $u$ is a vertex of $G$, then the neighbor of $u$ in $G$ is the set $N_G(u)=\{v:~v\in V(G),~d(u,v)=1\}$.
The degree of $u$ and the maximum degree of vertices of $G$ are given by $\deg(u)=|N_G(u)|$ and $\Delta(G)=\max\{\deg(v):~v\in V(G)\}$, respectively.

\begin{defi} {\rm \cite{XL}}
Let $f$ be a proper $k$-coloring of a connected graph $G$ and
$\Pi=(V_1,V_2,\ldots,V_k)$ be an ordered partition of $V(G)$ into
the resulting color classes. For a vertex $v$ of $G$, the {\bf color
code} of $v$ with respect to $\Pi$ is defined to be the ordered
$k$-tuple $$c_{{}_\Pi}(v)=(d(v,V_1),d(v,V_2),\ldots,d(v,V_k)).$$
  If distinct vertices of $G$ have distinct color codes, then $f$ is
called a {\bf locating coloring} of $G$. The {\bf locating chromatic number},
denoted by $\Cchi_{{}_L}(G)$,  is the minimum number of colors in a locating
coloring of $G$.
\end{defi}

The concept of locating coloring was first introduced and studied
by Chartrand et al. in \cite{XL}. They established some bounds for
the locating chromatic number of a connected graph. They also
proved that for a connected graph $G$ with $n\geq 3$ vertices, we
have $\Cchi_{_L}(G)=n$ if and only if $G$ is a complete
multipartite graph. Hence, the locating chromatic number of the
complete graph $K_n$ is $n$. Also for paths and cycles of order
$n\geq 3$ it is proved in \cite{XL} that $\Cchi_{_L}(P_n)=3$,
$\Cchi_{_L}(C_n)=3$ when $n$ is odd, and $\Cchi_{_L}(C_n)=4$ when
$n$ is even. \\The locating chromatic numbers of trees, Kneser
graphs, Cartesian product of graphs, and the amalgamation of stars are studied in \cite{XL},
\cite{Behtoei}, \cite{Behtoei2}, and \cite{Asmiati}, respectively. For more
results in the subject and related subjects, see~\cite{Asmiati}
to \cite{Conditional}.

Obviously, $\Cchi(G)\leq \Cchi_{{}_L}(G)$. Note that the $i$-th
coordinate of the color code of each vertex in the color class
$V_i$ is zero and its other coordinates are non zero. Hence, a
proper coloring is a locating coloring whenever the color codes of vertices
in each color class are different.

In this paper, we investigate the relation between the locating chromatic number of a tree and its maximum degree.
We provide a counter example to a theorem established in \cite{XL} about the locating chromatic numbers of trees. Then, we offer a new bound which is tight and best possible.

%%%%%%%%%%%%%%%%%%%%%%%%%%%%%%%%%%%%%%%%%%%%%%%%%%%%%%%%%%%%%%%%%%%%%
%%%%%%%%%%%%%%%%%%%%%%%%%%%%%%%%%%%%%%%%%%%%%%%%%%%%%%%%%%%%%%%%%%%%%
\section{Maximum degree and locating chromatic number}

For investigating the relation between the locating chromatic number of a tree and its maximum degree,
the following theorem is proved by Chartrand et al. \cite{XL}.

\begin{thm} {\rm \cite{XL}}\label{incorrectbound}
Let $k\geq 3$. If $~T$ is a tree for which $\Delta(T)>(k-1)2^{k-2}$, then $\Cchi_{_L}(T)>k$.
\end{thm}

They claimed that the bound in Theorem \ref{incorrectbound} is tight and cannot be improved. For this reason, in \cite{XL} they constructed a tree (for the special case $k=4$) with maximum degree $(4-1)2^{4-2}=12$ whose locating chromatic number is $4$.
\\
Actually, Theorem \ref{incorrectbound} is incorrect and its bound is not tight when $k$ is an integer greater than 4.
Fore instance, Theorem \ref{incorrectbound} implies that the locating chromatic number of each tree with maximum degree greater than $(5-1)2^{5-2}=32$ is at least $6$. But Example \ref{counterexample} provides a tree with maximum degree $36$ whose locating chromatic number is (at most) 5.

\begin{ex} \label{counterexample}
Consider the labeled tree $T_5$ which is illustrated in Figure \ref{fig:namedtree}. A proper vertex $5$-coloring of $T_5$, call it $f_5$, is illustrated in Figure \ref{fig:coloredtree}. For each $i$, $1\leq i\leq 5$, let $V_i$ be the set of vertices of $T_5$ with color $i$, and let $\Pi=(V_1,V_2,...,V_5)$. Color classes and corresponding color codes of the vertices of $T_5$ are illustrated in Table \ref{t1}. Since distinct vertices have distinct color codes, $f_5$ is a locating $5$-coloring of $T_5$ and hence, $\Cchi_{_L}(T_5)\leq 5$.
\end{ex}
\noindent Fore more details about the tree $T_5$ and its coloring $f_5$, see the arguments before Theorem \ref{construction}. Moreover, since
$$\Delta(T_5)=\deg(x)=36>4\times3^{4-3},$$ Theorem \ref{correctbound} implies that $\Cchi_{_L}(T_5)>4$. Therefore, $\Cchi_{_L}(T_5)=5$.
%%%%%%%%%%%%%%%%%%%%%%%%%%%%%%%%%%%%%%%%%%%%%%%%%%%%%
%%%%%%%%%%%%%%%%%%%%%%%%%%%%%%%%%%%%%%%%%%%%%%%%%%%%%%%%%%%%%%%%%%%%%
%%%%%%%%%%%%%%%%%%%%%%%%%%%%%%%%%%%%%%%%%%%%%%%%%%%%%%%%%%%%%%%%%%%%% picture 1
\begin{figure}\centering
%\unitlength=.52mm
\hspace*{-4mm}
\includegraphics[scale=0.99]{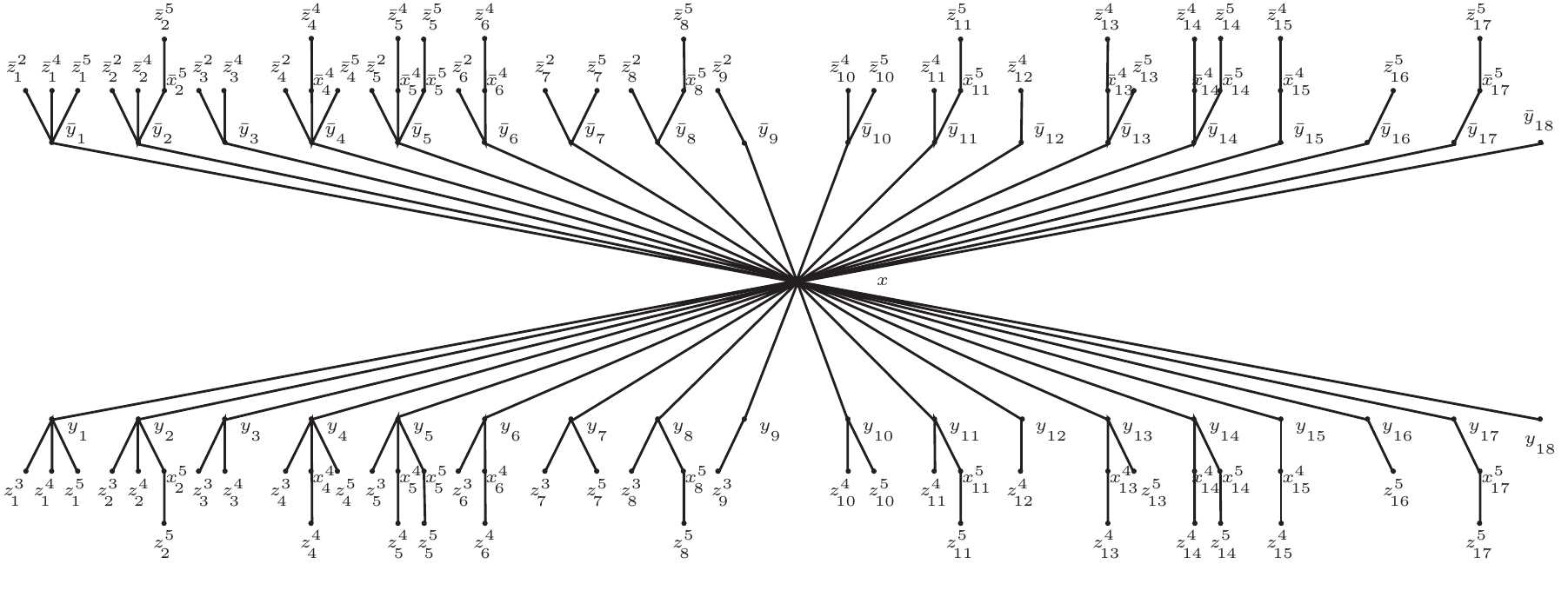}
%\vspace*{-8mm}
\caption{\label{fig:namedtree} The tree $T_5$ with maximum degree 36.}
\end{figure}
%%%%%%%%%%%%%%%%%%%%%%%%%%%%%%%%%%%%%%%%%%%%%%%%%%%%%%%%%%%%%%%%%%%%% picture 1

%%%%%%%%%%%%%%%%%%%%%%%%%%%%%%%%%%%%%%%%%%%%%%%%%%%%%%%%%%%%%%%%%%%%% picture 2
\begin{figure}\centering
%\setLTR
%\unitlength=.52mm
\hspace*{-5mm}
\includegraphics[scale=0.98]{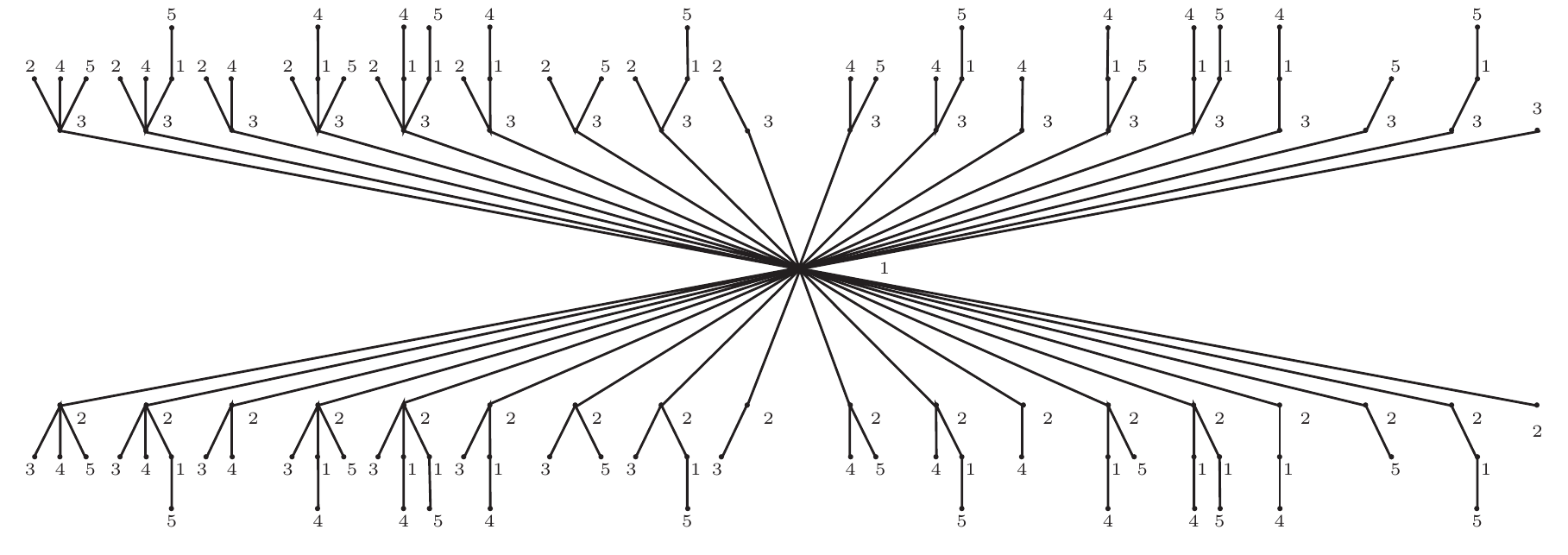}
%\setRTL
%\vspace*{-8mm}
\caption{\label{fig:coloredtree} The locating $5$-coloring $f_5$ of the tree $T_5$.}
\end{figure}
%%%%%%%%%%%%%%%%%%%%%%%%%%%%%%%%%%%%%%%%%%%%%%%%%%%%%%%%%%%%%%%%%%%%% picture 2
%\vspace{-7mm}
%%%%%%%%%%%%%%%%%%%%%%%%%%%%%%%%%%%%%%%%%%%%%%%%%%%%%%%%%%%%%%%%%%%%% Table 1
\begin{center}{
\begin{table}[h]\label{t1}
{\caption {\label{t1} The color classes and corresponding color codes of the locating $5$-coloring $f_5$ of $~T_5$.}}
%\begin{center}
\begin{tabular}{|c|c|c|c|c|}
\hline $V_1$ & $V_2$ & $V_3$ & $V_4$ & $V_5$\\ \hline
%%%%%%%%%%%%%%%%%%%%%%%%%%%%%%%%%%%%%%%%%%%%%%%%%%%%%%%%%%%%%%%%%%%%%%%%%%%%%%%%%%%%%%%%%%%%%%%%%%%%%%%%%%%
$x~(0,1,1,2,2)$&$y_{_{_{1}}}~(1,0,1,1,1)$&$\bar{y}_{_{_{1}}}~(1,1,0,1,1)$&$z_{_{\!1}}^4~(2,1,2,0,2)$&$z_{_{\!1}}^5~(2,1,2,2,0)$\\
$x_{_{\!2}}^5~(0,1,2,2,1)$&$y_{_{_{2}}}~(1,0,1,1,2)$&$\bar{y}_{_{_{2}}}~(1,1,0,1,2)$&$z_{_{\!2}}^4~(2,1,2,0,3)$&$z_{_{\!2}}^5~(1,2,3,3,0)$\\
$x_{_{\!4}}^4~(0,1,2,1,2)$&$y_{_{_{3}}}~(1,0,1,1,3)$&$\bar{y}_{_{_{3}}}~(1,1,0,1,3)$&$z_{_{\!3}}^4~(2,1,2,0,4)$&$z_{_{\!4}}^5~(2,1,2,3,0)$\\
$x_{_{\!5}}^4~(0,1,2,1,3)$&$y_{_{_{4}}}~(1,0,1,2,1)$&$\bar{y}_{_{_{4}}}~(1,1,0,2,1)$&$z_{_{\!4}}^4~(1,2,3,0,3)$&$z_{_{\!5}}^5~(1,2,3,4,0)$\\
$x_{_{\!5}}^5~(0,1,2,3,1)$&$y_{_{_{5}}}~(1,0,1,2,2)$&$\bar{y}_{_{_{5}}}~(1,1,0,2,2)$&$z_{_{\!5}}^4~(1,2,3,0,4)$&$z_{_{\!7}}^5~(2,1,2,4,0)$\\
$x_{_{\!6}}^4~(0,1,2,1,4)$&$y_{_{_{6}}}~(1,0,1,2,3)$&$\bar{y}_{_{_{6}}}~(1,1,0,2,3)$&$z_{_{\!6}}^4~(1,2,3,0,5)$&$z_{_{\!8}}^5~(1,2,3,5,0)$\\
$x_{_{\!8}}^5~(0,1,2,4,1)$&$y_{_{_{7}}}~(1,0,1,3,1)$&$\bar{y}_{_{_{7}}}~(1,1,0,3,1)$&$z_{_{\!10}}^4~(2,1,3,0,2)$&$z_{_{\!10}}^5~(2,1,3,2,0)$\\
$x_{_{\!11}}^5~(0,1,3,2,1)$&$y_{_{_{8}}}~(1,0,1,3,2)$&$\bar{y}_{_{_{8}}}~(1,1,0,3,2)$&$z_{_{\!11}}^4~(2,1,3,0,3)$&$z_{_{\!11}}^5~(1,2,4,3,0)$\\
$x_{_{\!13}}^4~(0,1,3,1,2)$&$y_{_{_{9}}}~(1,0,1,3,3)$&$\bar{y}_{_{_{9}}}~(1,1,0,3,3)$&$z_{_{\!12}}^4~(2,1,3,0,4)$&$z_{_{\!13}}^5~(2,1,3,3,0)$\\
$x_{_{\!14}}^4~(0,1,3,1,3)$&$y_{_{_{10}}}~(1,0,2,1,1)$&$\bar{y}_{_{_{10}}}~(1,2,0,1,1)$&$z_{_{\!13}}^4~(1,2,4,0,3)$&$z_{_{\!14}}^5~(1,2,4,4,0)$\\
$x_{_{\!14}}^5~(0,1,3,3,1)$&$y_{_{_{11}}}~(1,0,2,1,2)$&$\bar{y}_{_{_{11}}}~(1,2,0,1,2)$&$z_{_{\!14}}^4~(1,2,4,0,4)$&$z_{_{\!16}}^5~(2,1,3,4,0)$\\
$x_{_{\!15}}^4~(0,1,3,1,4)$&$y_{_{_{12}}}~(1,0,2,1,3)$&$\bar{y}_{_{_{12}}}~(1,2,0,1,3)$&$z_{_{\!15}}^4~(1,2,4,0,5)$&$z_{_{\!17}}^5~(1,2,4,5,0)$\\
$x_{_{\!17}}^5~(0,1,3,4,1)$&$y_{_{_{13}}}~(1,0,2,2,1)$&$\bar{y}_{_{_{13}}}~(1,2,0,2,1)$&$\bar{z}_{_{\!1}}^4~(2,2,1,0,2)$&$\bar{z}_{_{\!1}}^5~(2,2,1,2,0)$\\
$\bar{x}_{_{\!2}}^5~(0,2,1,2,1)$&$y_{_{_{14}}}~(1,0,2,2,2)$&$\bar{y}_{_{_{14}}}~(1,2,0,2,2)$&$\bar{z}_{_{\!2}}^4~(2,2,1,0,3)$&$\bar{z}_{_{\!2}}^5~(1,3,2,3,0)$\\
$\bar{x}_{_{\!4}}^4~(0,2,1,1,2)$&$y_{_{_{15}}}~(1,0,2,2,3)$&$\bar{y}_{_{_{15}}}~(1,2,0,2,3)$&$\bar{z}_{_{\!3}}^4~(2,2,1,0,4)$&$\bar{z}_{_{\!4}}^5~(2,2,1,3,0)$\\
$\bar{x}_{_{\!5}}^4~(0,2,1,1,3)$&$y_{_{_{16}}}~(1,0,2,3,1)$&$\bar{y}_{_{_{16}}}~(1,2,0,3,1)$&$\bar{z}_{_{\!4}}^4~(1,3,2,0,3)$&$\bar{z}_{_{\!5}}^5~(1,3,2,4,0)$\\
$\bar{x}_{_{\!5}}^5~(0,2,1,3,1)$&$y_{_{_{17}}}~(1,0,2,3,2)$&$\bar{y}_{_{_{17}}}~(1,2,0,3,2)$&$\bar{z}_{_{\!5}}^4~(1,3,2,0,4)$&$\bar{z}_{_{\!7}}^5~(2,2,1,4,0)$\\
$\bar{x}_{_{\!6}}^4~(0,2,1,1,4)$&$y_{_{_{18}}}~(1,0,2,3,3)$&$\bar{y}_{_{_{18}}}~(1,2,0,3,3)$&$\bar{z}_{_{\!6}}^4~(1,3,2,0,5)$&$\bar{z}_{_{\!8}}^5~(1,3,2,5,0)$\\
$\bar{x}_{_{\!8}}^5~(0,2,1,4,1)$&$\bar{z}_{_{\!1}}^2~(2,0,1,2,2)$&$z_{_{\!1}}^3~(2,1,0,2,2)$&$\bar{z}_{_{\!10}}^4~(2,3,1,0,2)$&$\bar{z}_{_{\!10}}^5~(2,3,1,2,0)$\\
$\bar{x}_{_{\!11}}^5~(0,3,1,2,1)$&$\bar{z}_{_{\!2}}^2~(2,0,1,2,3)$&$z_{_{\!2}}^3~(2,1,0,2,3)$&$\bar{z}_{_{\!11}}^4~(2,3,1,0,3)$&$\bar{z}_{_{\!11}}^5~(1,4,2,3,0)$\\
$\bar{x}_{_{\!13}}^4~(0,3,1,1,2)$&$\bar{z}_{_{\!3}}^2~(2,0,1,2,4)$&$z_{_{\!3}}^3~(2,1,0,2,4)$&$\bar{z}_{_{\!12}}^4~(2,3,1,0,4)$&$\bar{z}_{_{\!13}}^5~(2,3,1,3,0)$\\
$\bar{x}_{_{\!14}}^4~(0,3,1,1,3)$&$\bar{z}_{_{\!4}}^2~(2,0,1,3,2)$&$z_{_{\!4}}^3~(2,1,0,3,2)$&$\bar{z}_{_{\!13}}^4~(1,4,2,0,3)$&$\bar{z}_{_{\!14}}^5~(1,4,2,4,0)$\\
$\bar{x}_{_{\!14}}^5~(0,3,1,3,1)$&$\bar{z}_{_{\!5}}^2~(2,0,1,3,3)$&$z_{_{\!5}}^3~(2,1,0,3,3)$&$\bar{z}_{_{\!14}}^4~(1,4,2,0,4)$&$\bar{z}_{_{\!16}}^5~(2,3,1,4,0)$\\
$\bar{x}_{_{\!15}}^4~(0,3,1,1,4)$&$\bar{z}_{_{\!6}}^2~(2,0,1,3,4)$&$z_{_{\!6}}^3~(2,1,0,3,4)$&$\bar{z}_{_{\!15}}^4~(1,4,2,0,5)$&$\bar{z}_{_{\!17}}^5~(1,4,2,5,0)$\\
$\bar{x}_{_{\!17}}^5~(0,3,1,4,1)$&$\bar{z}_{_{\!7}}^2~(2,0,1,4,2)$&$z_{_{\!7}}^3~(2,1,0,4,2)$&$                 $&$ $\\
$                               $&$\bar{z}_{_{\!8}}^2~(2,0,1,4,3)$&$z_{_{\!8}}^3~(2,1,0,4,3)$&$                 $&$ $\\
$                               $&$\bar{z}_{_{\!9}}^2~(2,0,1,4,4)$&$z_{_{\!9}}^3~(2,1,0,4,4)$&$                 $&$ $\\
\hline
\end{tabular}
\end{table}
}\end{center}
%%%%%%%%%%%%%%%%%%%%%%%%%%%%%%%%%%%%%%%%%%%%%%%%%%%%%%%%%%%%%%%%%%%%%
%%%%%%%%%%%%%%%%%%%%%%%%%%%%%%%%%%%%%%%%%%%%%%%%%%%%%%%%%%%%%%%%%%%%% Table 1
\begin{lemma} \label{p23}
Let $k\geq 3$ be an integer. Then, $$\max\{p2^{p-1}3^q:~1+p+q=k,~p\in\Bbb{N},~q\in\Bbb{N}\cup\{0\}\}=4\times 3^{k-3}.$$
Moreover,  the maximum value happens just for $p\in\{2,3\}$.
\end{lemma}
\begin{proof}{
It can be easily checked that the statements of the theorem hold for $k\in\{3,4\}$. Hereafter suppose that $k\geq5$.
Since $1+p+q=k$, we have $q=k-1-p$. Therefore, If $p=2$, then $q=k-3$ and $p2^{p-1}3^q=4\times 3^{k-3}$.
Also, If $p=3$, then $q=k-4$ and $p2^{p-1}3^q=4\times 3^{k-3}$.
Now for each $p>3$ we have
\begin{eqnarray*}
(p-1)2^{p-2}3^{q+1}-p2^{p-1}3^q&=&3(p-1)-2p\\
&=&p-3>0.
\end{eqnarray*}
This means $(p-1)2^{p-2}3^{q+1}>p2^{p-1}3^q$, which completes the proof.
}\end{proof}

\begin{thm} \label{correctbound}
Let $k\geq 3$ be an integer. If $~T$ is a tree with $\Cchi_{_L}(T)=k$, then $\Delta(T)\leq 4\times 3^{k-3}$.
\end{thm}
\begin{proof}{
Assume that $T$ is a tree with locating chromatic number $k$ and $f:V(T)\rightarrow [k]$ is a locating $k$-coloring  of $T$.
Let $x$ be a vertex of maximum degree in $T$. Without loss of generality and by renaming the symbols $1,2,...,k$, if it is necessary, assume that $f(x)=1$ and $f(N_T(x))=\{2,3,...,p+1\}$ for some integer $p$, $1\leq p<k$.
For each $i$, $1\leq i\leq k$, $V_i=\{v:~v\in V(T),~f(v)=i\}$ is the $i$-th color class of $T$ and $\Pi=(V_1,V_2,...,V_k)$ is an ordered partition of $V(T)$ into the resulting color classes.
\\
The color code of $x$ is $c_{_{\Pi}}(x)=(d_1,d_2,...,d_k)$, where for each $j\in\{1,2,...,k\}$, $d_j=d(x,V_j)$.
Note that $d_1=0$, $d_2=d_3=\cdots=d_{p+1}=1$, and $d_j\geq 2$ for each $j>p+1$.
Let $i\in\{2,3,...,p+1\}$ be an arbitrary integer and $y\in V_i\cap N_T(x)$. The color code of $y$ is $c_{_{\Pi}}(y)=(c_1,c_2,...,c_k)$ where $c_1=1$ and $c_i=0$.
Also, for each $r\in\{2,3,...,p+1\}\setminus\{i\}$, $c_r\in\{1,2\}$, and  for each $j\in\{p+2,p+3,...,k\}$, $c_j\in\{d_j-1,d_j,d_j+1\}$.
Note that $$|\{2,3,...,p+1\}\setminus\{i\}|=p-1,~~|\{p+2,p+3,...,k\}|=k-1-p.$$
Since $f$ is a locating coloring, distinct vertices with (the same) color $i$ have distinct color codes. Hence,
\begin{eqnarray} \label{intersectsize}
\left|V_i\cap N_T(x)\right|=\left|\{c_{_{\Pi}}(y):~y\in V_i\cap N_T(x)\}\right|\leq 2^{p-1}3^{k-1-p}.
\end{eqnarray}
Inequality \ref{intersectsize} holds for each $i\in\{2,3,...,p+1\}$. Therefore,
\begin{eqnarray*}
\Delta(T)&=& |N_T(x)|\\
&=& |V_2\cap N_T(x)|+|V_3\cap N_T(x)|+...+|V_{p+1}\cap N_T(x)|\\
&\leq& 2^{p-1}3^{k-1-p}+2^{p-1}3^{k-1-p}+...+2^{p-1}3^{k-1-p} \\
&=& p2^{p-1}3^{k-1-p}.
\end{eqnarray*}
Now, by assuming that $q=k-1-p$, Lemma \ref{p23} implies $\Delta(T)\leq 4\times 3^{k-3}$.
}\end{proof}

Using Theorem \ref{correctbound}, if $\Delta(T)>4\times 3^{k-3}$, then $\Cchi_{_L}(T)>k$. This theorem shows that trees with large maximum degrees have large locating chromatic numbers.
Note that $$(3-1)2^{3-2}=4=4\times3^{3-3},~~(4-1)2^{4-2}=12=4\times3^{4-1}.$$
Thus, for $k\in\{3,4\}$ two bounds in Theorems \ref{incorrectbound} and \ref{correctbound} coincide.
But for each $k\geq5$ we have $(k-1)2^{k-2}<4\times3^{k-3}$ which means the bound in Theorem \ref{correctbound} is better. In fact, Theorem \ref{construction} shows that the bound provided by Theorem \ref{correctbound} is tight and is best possible. In Theorem \ref{construction} the structure of tree $T_5$ indicated in Example \ref{counterexample} is generalized.
\\Consider the tree $T_5$ and its coloring $f_5$ mentioned in Example \ref{counterexample}. Using the Cartesian product of sets, let
$$A=\{1\}\times\{0\}\times\{1,2\}\times\{1,2,3\}\times\{1,2,3\}=\{\alpha_{_1},\alpha_{_2},...,\alpha_{_{18}}\},$$
and
$$\bar{A}=\{1\}\times\{1,2\}\times\{0\}\times\{1,2,3\}\times\{1,2,3\}=\{\bar{\alpha}_{_1},\bar{\alpha}_{_2},...,\bar{\alpha}_{_{18}}\},$$
where the elements of $A$ and $\bar{A}$ are ordered using the lexicographic ordering. Therefore,
{\footnotesize
\begin{eqnarray*}
\alpha_{_1}=(1,0,1,1,1),~ \alpha_{_2}=(1,0,1,1,2),~ \alpha_{_3}=(1,0,1,1,3),~
\alpha_{_4}=(1,0,1,2,1),~ \alpha_{_5}=(1,0,1,2,2),~\alpha_{_6}=(1,0,1,2,3)~~~~~~~~ \\
\alpha_{_7}=(1,0,1,3,1),~ \alpha_{_8}=(1,0,1,3,2),~ \alpha_{_9}=(1,0,1,3,3),~ \alpha_{_{10}}=(1,0,2,1,1),~\alpha_{_{11}}=(1,0,2,1,2),~\alpha_{_{12}}=(1,0,2,1,3)~~~~~ \\
\alpha_{_{13}}=(1,0,2,2,1),~ \alpha_{_{14}}=(1,0,2,2,2),~ \alpha_{_{15}}=(1,0,2,2,3),~
\alpha_{_{16}}=(1,0,2,3,1),~ \alpha_{_{17}}=(1,0,2,3,2),~ \alpha_{_{18}}=(1,0,2,3,3)~~ \\
\bar{\alpha}_{_1}=(1,1,0,1,1),~\bar{\alpha}_{_2}=(1,1,0,1,2),~\bar{\alpha}_{_3}=(1,1,0,1,3),~
\bar{\alpha}_{_4}=(1,1,0,2,1),~\bar{\alpha}_{_5}=(1,1,0,2,2),~\bar{\alpha}_{_6}=(1,1,0,2,3)~~~~~~~~~\! \\
\bar{\alpha}_{_7}=(1,1,0,3,1),~\bar{\alpha}_{_8}=(1,1,0,3,2),~\bar{\alpha}_{_9}=(1,1,0,3,3),~
\bar{\alpha}_{_{10}}=(1,2,0,1,1),~\bar{\alpha}_{_{11}}=(1,2,0,1,2),~\bar{\alpha}_{_{12}}=(1,2,0,1,3)~~~~~ \\
\bar{\alpha}_{_{13}}=(1,2,0,2,1),~\bar{\alpha}_{_{14}}=(1,2,0,2,2),~\bar{\alpha}_{_{15}}=(1,2,0,2,3),~
\bar{\alpha}_{_{16}}=(1,2,0,3,1),~\bar{\alpha}_{_{17}}=(1,2,0,3,2),~\bar{\alpha}_{_{18}}=(1,2,0,3,3).~
\end{eqnarray*}
}

Using Table \ref{t1}, it can be easily checked that
$c_{_\Pi}(y_{_i})=\alpha_{_i}$ and $c_{_\Pi}(\bar{y}_{_i})=\bar{\alpha}_{_i}$ for each $i\in\{1,2,...,18\}$.
Note that $x$ is a vertex of maximum degree $4\times3^{5-3}=36$ in $T_5$ and $N_{T_5}(x)$ is divided into two disjoint sets $\{y_{_i}:~1\leq i\leq 18\}$ and $\{\bar{y}_{_i}:~1\leq i\leq 18\}$.
In fact,  the tree $T_5$ and its coloring $f_5$ are constructed using $5$-tuples in $A\cup \bar{A}$ in such a way that the color code of each $y_{_i}$ becomes $\alpha_{_i}$ and the color code of each $\bar{y}_{_i}$ becomes $\bar{\alpha}_{_i}$.
%The $5$-tuple $\alpha_i$ is assigned to the vertex $y_i$ and $\bar{\alpha}_i$ is assigned to the vertex $\bar{y}_i$, $1\leq i\leq 18$.
For instance, first component of each $\alpha_{_i}$ is $1$, and $x$ is a vertex with color $1$ which is adjacent to each $y_{_i}$. Second component of each $\alpha_{_i}$ is $0$ and the color of each  $y_{_i}$ is $2$.
Since $\alpha_{_1}=(1,0,1,1,1)$, vertex $y_{_1}$ is adjacent to vertices $z_{_1}^3$, $z_{_1}^4$, and $z_{_1}^5$ with colors $3$, $4$, and $5$, respectively.

Consider the vertex $y_{_{15}}$ and its corresponding $5$-tuple $\alpha_{_{15}}=(1,0,2,2,3)$. Third component of $\alpha_{_{15}}$ is $2$ and the distance from $y_{_{15}}$ to vertex $\bar{y}_{_1}$ with color three is $2$. Fourth component of $\alpha_{_{15}}$ is $2$ and the distance from $y_{_{15}}$ to vertex $z_{_{15}}^4$ with color four is $2$. Fifth component of $\alpha_{_{15}}$ is $3$ and the distance from $y_{_{15}}$ to vertex $z_{_1}^5$ with color five is $3$.
Now consider the vertex $z_{_{15}}^4$. The color of $z_{_{15}}^4$ is $4$ and this vertex is adjacent to $x_{_{15}}^4$ with color $1$. Hence, $d(z_{_{15}}^4,V_4)=0$ and $d(z_{_{15}}^4,V_1)=1$. For each vertex $v\in V(T_5)\setminus \{x_{_{15}}^4,z_{_{15}}^4\}$, the distance between $z_{_{15}}^4$ and $v$ in $T_5$ is given by $$d(z_{_{15}}^4,v)=d(z_{_{15}}^4,y_{_{15}})+d(y_{_{15}},v)=2+d(y_{_{15}},v).$$
Thus, $d(z_{_{15}}^4,V_i)=d(y_{_{15}},V_i)+2$ for each $i\in\{1,2,...,5\}\setminus\{1,4\}$.
Therefore, using the componentwise additions, subtractions and scalar multiplications, the color code of $z_{_{15}}^4$ is
\begin{eqnarray*}
c_{_\Pi}(z_{_{15}}^4)&=&\left(1,d(y_{_{15}},V_2)+2,d(y_{_{15}},V_3)+2,0,d(y_{_{15}},V_5)+2\right)\\
&=& \left(d(\bar{y}_{_{15}},V_1),d(y_{_{15}},V_2)+2,d(y_{_{15}},V_3)+2,d(y_{_{15}},V_4)-2,d(y_{_{15}},V_5)+2\right)\\
&=&
\left(d(\bar{y}_{_{15}},V_1),d(y_{_{15}},V_2),d(y_{_{15}},V_3),d(y_{_{15}},V_4),d(y_{_{15}},V_5)\right)
+(0,2,2,-2,2)\\
&=& c_{_\Pi}(y_{_{15}})+(2,2,2,2,2)-(2,0,0,4,0)\\
&=& \alpha_{_{15}}+2(1,1,1,1,1)-2(1,0,0,0,0)-4(0,0,0,1,0).
\end{eqnarray*}
Similarly, the color of vertex $z_{_{4}}^3$ is $3$ and
\begin{eqnarray*}
c_{_\Pi}(z_{_{4}}^3)&=&\left(d(y_{_{4}},V_1)+1,d(y_{_{4}},V_2)+1,0,d(y_{_{4}},V_4)+1,d(y_{_{4}},V_5)+1\right)\\
&=&\left(d(y_{_{4}},V_1)+1,d(y_{_{4}},V_2)+1,d(y_{_{4}},V_3)-1,d(y_{_{4}},V_4)+1,d(y_{_{4}},V_5)+1\right)\\
&=&\left(d(y_{_{4}},V_1),d(y_{_{4}},V_2),d(y_{_{4}},V_3),d(y_{_{4}},V_4),d(y_{_{4}},V_5)\right)+(1,1,-1,1,1)\\
&=& \alpha_{_{4}}+(1,1,1,1,1)-2(0,0,1,0,0).
\end{eqnarray*}
Similar arguments hold for other vertices of $T_5$.
%These ideas are generalized in Theorem \ref{construction}.

\begin{thm}  \label{construction}
For each integer $k\geq 3$, there exists a tree $T_k$ with maximum degree $4\times 3^{k-3}$ whose locating chromatic number is $k$.
\end{thm}
\begin{proof}{
For $k\in \{3,4\}$ two trees $T_3$ and $T_4$ with their optimum locating colorings are shown in Figure \ref{fig:T3T4}.
%%%%%%%%%%%%%%%%%%%%%%%%%%%%%%%%%%%%%%%%%%%%%%%%%%%%%%%%%%%%%%%%%%%%% picture T3 and T4
%%%%%%%%%%%%%%%%%%%%%%%%%%%%%%%%%%%%%%%%%%%%%%%%%%%%%%%%%%%%%%%%%%%%%
\begin{figure}\centering
%\setLTR
%\unitlength=.55mm
%\vspace*{-5mm}
%\hspace*{-22mm}
\includegraphics[scale=1]{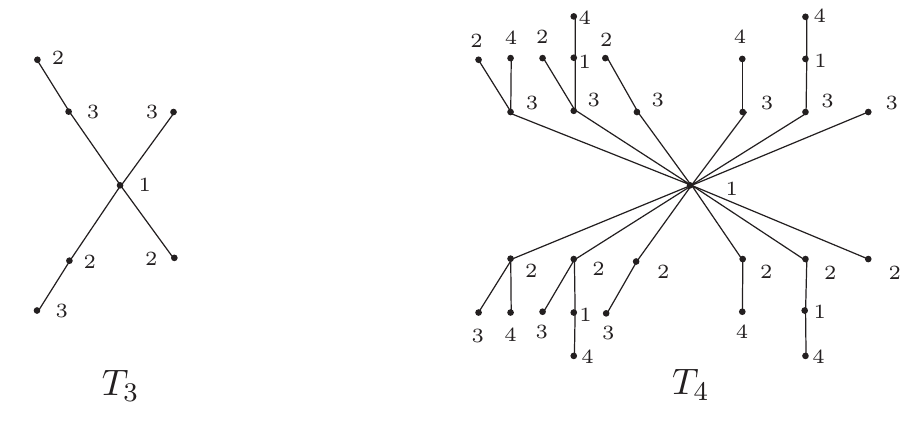}
%\setRTL
\caption{\label{fig:T3T4} Two trees $T_3$ and $T_4$ with their optimum locating colorings.}
\end{figure}
%%%%%%%%%%%%%%%%%%%%%%%%%%%%%%%%%%%%%%%%%%%%%%%%%%%%%%%%%%%%%%%%%%%%%
%%%%%%%%%%%%%%%%%%%%%%%%%%%%%%%%%%%%%%%%%%%%%%%%%%%%%%%%%%%%%%%%%%%%% picture T3 and T4
Hereafter suppose that $k\geq 5$. The tree $T_5$ is investigated in Example \ref{counterexample}. We want to construct $T_k$ using the ideas involved in the structure of $T_5$.
Let
$${\cal A}=\{1\} \times \{0\} \times \{1,2\} \times \{1,2,3\}^{k-3},
~~~\bar{{\cal A}}=\{1\} \times \{1,2\} \times \{0\} \times \{1,2,3\}^{k-3}.$$
%and
%$$Y=\{y_{\alpha}:~\alpha\in {\cal A}\},~~~Z=\{z_{\beta}:~\beta\in {\cal B}\}.$$
Note that each element of ${\cal A}$ and $\bar{{\cal A}}$ is a $k$-tuple and $|{\cal A}|=|\bar{{\cal A}}|=2\times 3^{k-3}$. Also, note that  ${\cal A}$ and $\bar{{\cal A}}$ are two disjoint sets. If $\gamma \in {\cal A}\cup\bar{{\cal A}}$ and $1\leq t\leq k$, then by $\gamma(t)$ we mean the $t$-th coordinate of $\gamma$. Assume that ${\cal A}=\{\alpha_1,\alpha_2,...,\alpha_{_{|{\cal A}|}}\}$ and $\bar{{\cal A}}=\{\bar{\alpha}_1,\bar{\alpha}_2,...,\bar{\alpha}_{_{|\bar{{\cal A}}|}}\}$ where the elements in ${\cal A}$ and $\bar{{\cal A}}$ are ordered using the lexicographic ordering.
%For example, when $k=5$ then we have ?????\\
Hence, $$\alpha_1=(1,0,1,1,1,...,1),~\alpha_{|{\cal A}|}=(1,0,2,3,3,...,3),~~~\bar{\alpha}_1=(1,1,0,1,1,...,1),~\bar{\alpha}_{|\bar{{\cal A}}|}=(1,2,0,3,3,...,3).$$
Let $T_k$ be the tree with vertex set $V(T_k)$ and edge set $E(T_k)$ defined as follow.
\begin{eqnarray*}
V(T_k)&=&\{x\} \cup \{y_i:~ 1\leq i\leq |{\cal A}|\} \cup \{\bar{y}_i:~1\leq i\leq |\bar{{\cal A}}|\} \\
&~&\cup~ \{z_i^3:~1\leq i\leq |{\cal A}|,~\alpha_i(3)=1 \} \cup \{\bar{z}_i^2:~1\leq i\leq |\bar{{\cal A}}|,~\bar{\alpha}_i(2)=1 \}\\
&~&\cup~ \{z_i^t:~1\leq i\leq |{\cal A}|,~4\leq t\leq k,~1\leq\alpha_i(t)\leq2 \} \cup \{\bar{z}_i^t:~1\leq i\leq |\bar{{\cal A}}|,~4\leq t\leq k,~1\leq \bar{\alpha}_i(t)\leq 2 \}\\
&~&\cup~ \{x_i^t:~1\leq i\leq |{\cal A}|,~4\leq t\leq k,~\alpha_i(t)=2 \} \cup \{\bar{x}_i^t:~1\leq i\leq |\bar{{\cal A}}|,~4\leq t\leq k,~\bar{\alpha}_i(t)=2 \},
\end{eqnarray*}
and
\begin{eqnarray*}
E(T_k)&=&\{xy_i:~1\leq i\leq |{\cal A}|\} \cup \{x\bar{y}_i:~1\leq i\leq |\bar{{\cal A}}|\} \\
&~& \cup~ \{y_i z_i^t:~1\leq i\leq |{\cal A}|,~2\leq t\leq k,~\alpha_i(t)=1\} \cup
\{\bar{y}_i \bar{z}_i^t:~1\leq i\leq |\bar{{\cal A}}|,~2\leq t\leq k,~\bar{\alpha}_i(t)=1\} \\
&~& \cup~ \{y_i x_i^t:~1\leq i\leq |{\cal A}|,~4\leq t\leq k,~\alpha_i(t)=2\} \cup \{\bar{y}_i \bar{x}_i^t:~1\leq i\leq |\bar{{\cal A}}|,~4\leq t\leq k,~\bar{\alpha}_i(t)=2\} \\
&~& \cup~ \{x_i^t z_i^t:~1\leq i\leq |{\cal A}|,~4\leq t\leq k,~\alpha_i(t)=2 \} \cup \{\bar{x}_i^t \bar{z}_i^t:~1\leq i\leq |\bar{{\cal A}}|,~4\leq t\leq k,~\bar{\alpha}_i(t)=2 \}.
\end{eqnarray*}
For the case $k=5$, the tree $T_5$ is illustrated in Figure \ref{fig:namedtree}.
Note that $$\Delta(T_k)=\deg(x)=|{\cal A}|+|\bar{{\cal A}}|=2\times3^{k-3}+2\times3^{k-3}=4\times3^{k-3}>4\times3^{(k-1)-3}.$$
Thus, Theorem \ref{correctbound} implies that the locating chromatic number of $T_k$ is at least $k$. We show that $\Cchi_{_L}(T_k)$ is equal to $k$. For this reason, it is sufficient to provide a locating $k$-coloring of the vertices of $T_k$. \\
%We want to give a coloring of the vertices of $T_k$ in such a way that the color code of vertices of the form $y_i$ and $\bar{y}_i$ become $\alpha_i$ and $\beta_i$, respectively. \\
Define the vertex $k$-coloring $f_k:V(T_k)\rightarrow [k]$ as follow.
\begin{eqnarray*}
f_k(v)=\left\{\begin{array}{ll}
1 & v=x  \\
2 & v=y_i,~1\leq i\leq |{\cal A}| \\
3 & v=\bar{y}_i,~1\leq i\leq |\bar{{\cal A}}|  \\
3 & v=z_i^3,~1\leq i\leq |{\cal A}|,~\alpha_i(3)=1  \\
2 & v=\bar{z}_i^2,~1\leq i\leq |\bar{{\cal A}}|,~\bar{\alpha}_i(2)=1  \\
1 & v=x_i^t,~1\leq i\leq |{\cal A}|,~4\leq t\leq k,~\alpha_i(t)=2 \\
1 & v=\bar{x}_i^t,~1\leq i\leq |\bar{{\cal A}}|,~4\leq t\leq k,~\bar{\alpha}_i(t)=2\\
t & v=z_i^t,~1\leq i\leq |{\cal A}|,~4\leq t\leq k,~1\leq\alpha_i(t)\leq2 \\
t & v=\bar{z}_i^t,~1\leq i\leq |\bar{{\cal A}}|,~4\leq t\leq k,~1\leq \bar{\alpha}_i(t)\leq2.
\end{array}\right.
\end{eqnarray*}
For the case $k=5$, the coloring $f_5$ of $T_5$ is illustrated in Figure \ref{fig:coloredtree}.
%??  the corresponding color classes are as follow  ??
Let $\Pi=(V_1,V_2,...,V_k)$ be an ordered partition of $V(T_k)$ into the resulting color classes $V_i=\{v:~v\in V(T_k),~f_k(v)=i\}$, $1\leq i\leq k$.
\\ For each $i$, $1\leq i\leq k$, let $e_i=(\delta_{1,i},\delta_{2,i},...,\delta_{k,i})$ where $\delta_{j,i}$ is the Kroneker delta. Also, using the componentwise summation, let $e=e_1+e_2+\cdots+e_k$. Therefore, $$e_1=(1,0,0,...,0,0),~e_2=(0,1,0,...,0,0),~...,~e_k=(0,0,0,..,0,1),~e=(1,1,1,...,1,1).$$
Using the structure of $T_k$ and by a simple argument similar to what is discussed after Theorem \ref{correctbound}, we can easily derive the following color codes.
\begin{eqnarray*}
c_{_{\Pi}}(v)=\left\{\begin{array}{ll}
(0,1,1,2,2,...,2) & v=x  \\
\alpha_{_i} & v=y_i,~1\leq i\leq |{\cal A}| \\
\alpha_{_i}+e-2e_t & v=z_i^t,~1\leq i\leq |{\cal A}|,~2\leq t\leq k,~\alpha_i(t)=1 \\
\alpha_{_i}+e-2e_1-2e_t & v=x_i^t,~1\leq i\leq |{\cal A}|,~4\leq t\leq k,~\alpha_i(t)=2 \\
\alpha_{_i}+2e-2e_1-4e_t & v=z_i^t,~1\leq i\leq |{\cal A}|,~4\leq t\leq k,~\alpha_i(t)=2 \\
\bar{\alpha}_{_i} & v=\bar{y}_i,~1\leq i\leq |\bar{{\cal A}}| \\
\bar{\alpha}_{_i}+e-2e_t & v=\bar{z}_i^t,~1\leq i\leq |\bar{{\cal A}}|,~2\leq t\leq k,~\bar{\alpha}_i(t)=1 \\
\bar{\alpha}_{_i}+e-2e_1-2e_t & v=\bar{x}_i^t,~1\leq i\leq |\bar{{\cal A}}|,~4\leq t\leq k,~\bar{\alpha}_i(t)=2 \\
\bar{\alpha}_{_i}+2e-2e_1-4e_t & v=\bar{z}_i^t,~1\leq i\leq |\bar{{\cal A}}|,~4\leq t\leq k,~\bar{\alpha}_i(t)=2. \\
\end{array}\right.
\end{eqnarray*}
Therefore, distinct vertices of $T_k$ have distinct color codes. Hence, $f_k$ is a locating $k$-coloring of $T_k$ and this completes the proof.
}\end{proof}
{\bf Acknowledgements.}
We would like to express our deepest gratitude to Dr. Behnaz Omoomi for her invaluable comments and suggestions.
This research was in part supported by a grant from IPM (No. 91050012).
%%%%%%%%%%%%%%%%%%%%%%%%%%%%%%%%%%%%%%%%%%%%%%%%%%%%%%%%%%%%%%%%%%%%%%%%%%%%%%%%%%%%%%%%%%%%%%%%%%%%%%%%%%%%%%%%%%%%%%%%%%%%%%


\begin{thebibliography}{99}
\bibitem{Asmiati}{\sc Asmiati, H. Assiyatun,} and {\sc E.T. Baskoro}, Locating-chromatic number of amalgamation of stars,
{\it ITB J. Sci.} {\bf 43, A} (2011) 1-8.

\bibitem{Behtoei2}{\sc A. Behtoei} and  {\sc B. Omoomi}, On the locating chromatic number of the cartesian product of graphs,
{\it to appear in Ars Combin.}

\bibitem{Behtoei}{\sc A. Behtoei} and  {\sc B. Omoomi}, On the locating chromatic number of Kneser graphs,
{\it Discrete  Appl. Math.} {\bf 159} (2011) 2214-2221.

%\bibitem{Amalgamation}{\sc K. Carlson}, Generalized books and $C_m$-snakes are prime graphs,
%{\it Ars Combin.} {\bf 80} (2006) 215-221.

\bibitem{XL}{\sc G. Chartrand, D. Erwin, M.A. Henning, P.J. Slater,} and {\sc P. Zhang}, The locating-chromatic number of a graph,
{\it Bull. Inst. Combin. Appl.} {\bf 36} (2002) 89-101.

\bibitem{XLn-1}{\sc G. Chartrand, D. Erwin, M.A. Henning, P.J. Slater,} and {\sc P. Zhang}, Graphs of order $n$ with locating-chromatic number $n-1$,
{\it Discrete Math.} {\bf no. 1–3, 269} (2003)  65-79.

\bibitem{Metric09}{\sc G. Chartrand, F. Okamoto,} and {\sc P. Zhang}, The metric chromatic number of a graph,
{\it Australasian Journal of Combinatorics} {\bf 44 } (2009)
273-286.

\bibitem{ResolvingEdg}{\sc G. Chartrand, V. Saenpholphat,} and {\sc P. Zhang},  Resolving edge colorings in graphs, {\it  Ars Combin.} {\bf 74} (2005) 33-47.

\bibitem{partition}{\sc G. Chartrand, E. Salehi,} and {\sc P. Zhang}, The partition dimension of a graph, {\it Aequationes Math.} {\bf no. 1-2, 59} (2000)  45-54.

%\bibitem{Grahambp}{\sc R.L. Graham,  and H.O. Pollak}, On the addressing problem for loop switching, {\it %Bell System Techn. Jour.} {\bf 50} (1971) 2495-2519.

\bibitem{Harary}{\sc F. Harary,} and  {\sc R.A. Melter}, On the metric dimension of a graph, {\it Ars Combin.} {\bf 2} (1976) 191-195.

%\bibitem{Kneser}{\sc M. Kneser}, Aufgabe 360, {\it Jahresbericht der Deutschen Mathematiker-Vereinigung.} %{\bf 58} (1955) 27.

\bibitem{Conditional}{\sc V. Saenpholphat} and {\sc P. Zhang}, Conditional resolvability in graphs: A survey, {\it Int. J. Math. Sci.} {\bf 37-40} (2004) 1997-2017.

%\bibitem{ResolvingDec}{\sc V. Saenpholphat and P. Zhang}, On connected resolving decompositions in graphs, %{\it Czechoslovak Math. J.} {\bf 54} (2004) 681-696.

%\bibitem{Slater}{\sc P.J. Slater}, Leaves of trees, {\it Congress. Numer.} {\bf 14} (1975) 549-559.

\end{thebibliography}
\end{document}